\documentclass[10pt]{article}
\title{On Debreu-Koopmans Theorem for Additively Decomposed Quasiconvex Functions with Applications}
\author{Felipe Lara\thanks{Instituto de Alta investigaci\'on (IAI), Universidad de Tarapac\'a, Arica, Chile. E-mail: felipelaraobreque@gmail.com;
flarao@academicos.uta.cl. Web: felipelara.cl, ORCID-ID: 0000-0002-9965-0921}}

\usepackage{amsmath}
\usepackage{amsthm}
\usepackage{amssymb}
\usepackage{multicol}
\usepackage[active]{srcltx}
\usepackage{algorithm}
\usepackage{array}
\usepackage[dvipsnames,svgnames]{xcolor}
\usepackage{url}
\usepackage{tikz}
\usepackage{subfig}
\usepackage{graphicx}
\usepackage{pgfplots}
\usepackage{xcolor}
\usepackage{ulem}  
\usepackage{comment}
\providecommand{\U}[1]{\protect \rule{.1in}{.1in}}
\newtheorem{theorem}{Theorem}

\newtheorem{corollary}[theorem]{Corollary}

\newtheorem{definition}[theorem]{Definition}
\newtheorem{example}[theorem]{Example}

\newtheorem{proposition}[theorem]{Proposition}
\newtheorem{remark}[theorem]{Remark}

\begin{document}


\maketitle


\begin{abstract}
\noindent 
The Debreu–Koopmans theorem \cite{DK} restricts separable aggregation to at most one nonconvex component. We solve this by proving that a separable (additive or multiplicative) function is star quasiconvex (those with star-shaped sublevel sets about minimizers) if and only if each component is star quasiconvex. This immediately yields star quasiconvexity of separable sums of quasiconvex functions, formally bridging diversification theory with the $S$-shaped value functions of Prospect Theory. Furthermore, we develop a complete calculus (monotonic composition, pointwise minima, quasi-arithmetic means) and we apply it to Cobb-Douglas functions, multi-factor risk models, and constant function market makers in decentralized finance. Star quasiconvexity thus provides a unified framework for applications in optimization and economic modeling beyond the classical Debreu–Koopmans constraint. The introduction discuss economic motivations.
{\small}

\medskip


\medskip

\end{abstract}

\section{Introduction}

A fundamental challenge in economic modeling is reconciling realistic behavioral assumptions with mathematical tractability. Among the most basic normative axioms is the {\it tendency to diversification} (or convex preferences): economic agents, whether consumers or investors, generally prefer balanced combinations of goods or assets to extreme bundles. Mathematically, this principle is formalized through {\it quasiconcavity} of utility, production, or social welfare functions (see, for instance, \cite{D-1959}).

To manage analytical complexity in multi-dimensional settings, economists frequently employ {\it separable} (additively or multiplicatively decomposed) functional forms, where an aggregate outcome is expressed as a sum of independent components. This structure is pervasive in applications ranging from consumer theory and production economics to macroeconomic growth and modeling \cite{AE,Green,Slutsky,Yaa}, and it facilitates computation through efficient methods like block-coordinate optimization \cite{BT,N-2012,W-2015}, widely used in large-scale optimization and machine learning among others.

Consider a convex set $X \subseteq \mathbb{R}^{n}$ such that $X = \prod^{m}_{i=1} X_{i}$, where each $X_{i} \subseteq \mathbb{R}^{n_{i}}$ is convex and $\sum^{m}_{i=1} n_{i} = n$. Let $h_{i}: X_{i} \rightarrow \mathbb{R}$ be functions and $h: X \rightarrow \mathbb{R}$ be the function defined by
\begin{equation}\label{separable}
 h(x_{1}, \ldots, x_{m}) := \sum \limits_{i=1}^{m} h_{i}(x_{i}). 
\end{equation}
For (strongly) convex functions, a complete characterization is known:
\begin{equation}\label{char:convex}
 h \text{ is (strongly) convex} ~ \Longleftrightarrow ~ \text{every } h_{i} \text{ is (strongly) convex}.
\end{equation}
However, for the economically crucial class of quasiconvex functions (a function $h$ is quasiconvex if and only if $-h$ is quasiconcave), which is related to diversification, no analogous equivalence holds. Instead, the Debreu-Koopmans theorem \cite[Theorems 2 and 10]{DK} (see also \cite{CL}) imposes a severe restriction: if a separable function \(h\) is quasiconvex, then {\it at most one} component \(h_i\) may be nonconvex.

This result creates a profound methodological dilemma: to maintain diversification (quasiconvexity) within a separable model, one must impose convexity on nearly all components. In the Bergson-Samuelson social welfare function \cite{Berg,Samuel}, for instance, this forces almost all individuals to have convex utility functions, severely restricting the range of admissible preference patterns. Similarly, in ratio models with multiple risk factors modeled via fractional programming, the Debreu-Koopmans theorem permits at most one risk variable to be active (see, for instance, \cite[Chapter 5]{ADSZ}), an unrealistic limitation in virtually any applied economic setting. 

This leaves open a fundamental, and inverse, question: If we start from the economically natural assumption that each component is quasiconvex (representing local diversification), what can be said about the global property of their sum?. We formulate this as the {\it Debreu-Koopmans problem}:
\begin{equation}\label{DK}
\text{If every } h_{i} \text{ is quasiconvex} ~ \Longrightarrow ~ h = \sum^{m}_{i=1} h_i \in ~ ?. \tag{DK}
\end{equation}
This paper solves the Debreu-Koopmans problem by showing that the separable aggregation of quasiconvex components yields a {\it star quasiconvex} function, one whose sublevel sets are star-shaped with respect to its minimizers \cite{KL,NTV}. This class is strictly broader than quasiconvexity (every quasiconvex function with a minimizer is star quasiconvex, but the converse  fails), yet retains an economically meaningful structure that aligns with  {\it star-shaped preferences} centered at a minimum. Beyond additive aggregation, we extend this characterization to multiplicative separability and develop a comprehensive calculus encompassing monotonic composition, pointwise minima, and {\it weighted quasi-arithmetic means}, providing a versatile toolkit for economic and optimization modeling.

\medskip

As direct consequences of these results:
\begin{itemize}
 \item[$(i)$] We prove that separable $S$-shaped value functions (concave for gains and convex for losses), as they appear in Kahneman and Tversky's Prospect Theory \cite{KT,TK}, are star quasiconvex, formally bridging normative diversification with behavioral choice.
    
 \item[$(ii)$] We show that workhorse production and utility functions, including CES (Constant Elasticity of Substitution), Cobb-Douglas and Leontief, are star quasiconvex on compact domains. 
 
 \item[$(iii)$] We demonstrate that ratio models with multiple risk factors (where benefits and risks are modeled separately as $f(x)/g(y)$) become tractable under star quasiconvexity. While the Debreu-Koopmans theorem restricts such models to at most one risk variable, star quasiconvexity imposes no such limitation, enabling realistic specifications with multiple sources of uncertainty.
\end{itemize}


The structure of the paper is as follows. Section \ref{sec:02} recalls the necessary preliminaries on star quasiconvexity and its geometric interpretation. Section \ref{sec:03} contains our main results: we characterize separable star quasiconvex functions, resolving the Debreu-Koopmans problem, and show that separable $S$-shaped value functions from Prospect Theory are star quasiconvex, formally bridging neoclassical and behavioral economics. We also develop a calculus rule for star quasiconvexity with applications to production and utility functions (Leontief, Cobb-Douglas). Section \ref{sec:04} demonstrates the optimization/economic relevance of our framework through weighted quasi-arithmetic means and ratio models with multiple risk factors. Section \ref{sec:05} concludes with a discussion of the implications for future research in both economic theory and mathematical optimization.

\section{Preliminaries and Basic Definitions} \label{sec:02}

The inner product in $\mathbb{R}^{n}$ and the Euclidean norm are denoted by $\langle \cdot,\cdot \rangle$ and $\lVert \cdot \rVert$, respectively. We denote  $\mathbb{R}_{+} := [0, + \infty[$ and $ \mathbb{R}_{++} := \, ]0, + \infty[$, thus $\mathbb{R}^{n}_{+} = [0, + \infty[ \, \times \ldots \times [0, + \infty[$ and $\mathbb{R}^{n}_{++} = \, ]0, + \infty[ \, \times \ldots \times \, ]0, + \infty[$ ($n$ times). 

Let $K \subseteq \mathbb{R}^{n}$ be nonempty. Then given $x_{0} \in K$, the set $K$ is said to be {\it star-shaped} at the point $x_{0}$ if for every $x \in K$, we have $[\overline{x}, x] \subseteq K$. 

Given a real-valued function $h: \mathbb{R}^{n} \rightarrow \mathbb{R}$, it is indicated by ${\rm epi}\,h := \{(x,t) \in \mathbb{R}^{n} \times \mathbb{R}: h(x) \leq t\}$ the epigraph of $h$, by $S_{\delta} (h) := \{x \in \mathbb{R}^{n}: h(x) \leq \delta\}$ the sublevel set of $h$ at the height $\delta \in \mathbb{R}$ and by ${\rm argmin}_{ \mathbb{R}^{n}} h$ (resp. ${\rm argmax}_{ \mathbb{R}^{n}} h$) the set of all minimal (resp. maximal) points of $h$ on $\mathbb{R}^{n}$. 

Throughout this paper, given a function $g: \mathbb{R} \rightarrow \mathbb{R}$, by {\it increasing} (resp. {\it decreasing}) we mean $a < b$ implies $g(a) < g(b)$ (resp. $a < b$ implies $g(a) > g(b)$), and by {\it nondecreasing} (resp. {\it nonincreasing}) we mean $a < b$ implies $g(a) \leq g(b)$ (resp. $a < b$ implies $g(a) \geq g(b)$). Also, we say that $g$ is {\it strictly monotonic} if it is either increasing or decreasing.

A function $h: \mathbb{R}^{n} \rightarrow \mathbb{R}$ is said to be:
\begin{itemize}
 \item[$(a)$] (strongly) convex if there exists $\gamma \geq 0$ such that if for all $x, y \in \mathbb{R}^{n}$ and all $\lambda \in [0, 1]$, we have
 \begin{equation}\label{strong:convex}
  h(\lambda y + (1-\lambda)x) \leq \lambda h(y) + (1-\lambda) h(x) - \lambda (1 - \lambda) \frac{\gamma}{2} \lVert x - y \rVert^{2},
 \end{equation}

 \item[$(b)$] (strongly) quasiconvex if there exists $\gamma \geq 0$ such that if for all $x, y \in \mathbb{R}^{n}$ and all $\lambda \in [0, 1]$, we have
 \begin{equation}\label{strong:quasiconvex}
  h(\lambda y + (1-\lambda)x) \leq \max \{h(y), h(x)\} - \lambda(1 - \lambda) \frac{\gamma}{2} \lVert x - y \rVert^{2}.
 \end{equation}
\end{itemize} 
\noindent Every (strongly) convex function is (strongly) quasiconvex, while the reverse statement does not holds (see \cite{ADSZ,CM-Book,Lara-9}). Furthermore, when we say that a function is {\it strongly convex (resp. quasiconvex) with modulus $\gamma \geq 0$} we refer to both strongly convex (resp. quasiconvex) when $\gamma > 0$ and convex (resp. qua\-siconvex) when $\gamma = 0$.

As we explained in the introduction, {\it quasiconvexity} is related to the mathematical formulation of the assumption {\it tendency to diversification} on the consumers in consumer's preference theory (see Debreu \cite{D-1959}). Furthermore, the following geometric characterizations are well-known:
\begin{align*}
  h ~ \mathrm{is ~ convex} & \Longleftrightarrow \, \mathrm{epi}\,h ~ \mathrm{is ~ a ~ convex ~ set,} \\ 
  h ~ \mathrm{is ~ quasiconvex} & \Longleftrightarrow \, S_{\delta} (h) ~
 \mathrm{is ~ a ~ convex ~ set ~ for ~ all ~ } \delta \in \mathbb{R}.
\end{align*}



Now, we recall the following definition from  \cite{KL,NTV}, which has been developed in virtue of its interesting properties for ensuring linear convergence of first-order type methods in optimization.

\begin{definition}\label{ss:quasiconvex}
 Let $h: \mathbb{R}^{n} \rightarrow \mathbb{R}$ be  function and $\overline{x} \in {\rm argmin}_{\mathbb{R}^{n}}\,h$. Then $h$
 is (strongly) star quasiconvex with modulus $\gamma \geq 0$ with respect to $\overline{x}$ if 
 \begin{align*}\label{ss:qcx}
  h(\lambda \overline{x} + (1-\lambda)y) \leq h(y) - \lambda (1 - \lambda)
  \frac{\gamma}{2} \lVert y - \overline{x} \rVert^{2}, \, \forall \, \lambda \in 
  [0, 1], \, \forall \, y \in \mathbb{R}^{n}.
 \end{align*}
 A function $h$ is strongly star quasiconcave if $-h$ is strongly star quasiconvex. 
\end{definition}
We note immediately that when $\gamma > 0$, strongly star quasiconvex functions has an unique minimizer, i.e., in this case, we simple say that $h$ is strongly star quasiconvex with modulus $\gamma > 0$ (without refering to a specific point $\overline{x}$). 

Three important geometric properties are described below.  Let $h: \mathbb{R}^{n} \rightarrow \mathbb{R}$ be a function. Then the following assertions hold:
\begin{itemize} 
 \item[$(i)$] (see \cite[Theorem 9]{KL}) $h$ is star quasiconvex at $\overline{x} \in {\rm argmin}_{\mathbb{R}^{n}}\,h$ if and only if $S_{\delta} (h)$ is star-shaped  at $\overline{x}$ for every $\delta \in \mathbb{R}$.

 \item[$(ii)$] (see \cite[Theorem 12]{KL}) Let $\overline{x} \in {\rm argmin}_{\mathbb{R}^{n}}\,h$. Then, $h$ is (strongly) star quasiconvex at $\overline{x}$ with modulus $\gamma \geq 0$ if and only if its restriction on any half line through $\overline{x}$ is (strongly) qua\-siconvex with the same modulus $\gamma \geq 0$.

 \item[$(iii)$] {\rm see (\cite[Theorem 3.3]{NTV})}:
 Let $h: \mathbb{R}^{n} \rightarrow \mathbb{R}$ be a di\-ffe\-ren\-tia\-ble function and $\overline{x} \in {\rm argmin}_{\mathbb{R}^{n}}\,h$. Then, $h$ is (strongly) star qua\-si\-con\-vex with modulus $\gamma \geq 0$ with respect to $\overline{x}$ if and only if for every $y \in \mathbb{R}^{n}$, we have
 \begin{equation}\label{char:grad}
  \langle \nabla h(y), y - \overline{x} \rangle \geq \frac{\gamma}{2} \lVert y - \overline{x} \rVert^{2}.
 \end{equation}
\end{itemize} 

\begin{remark}\label{rem:relat}
 \begin{itemize}
  \item[$(i)$] The following re\-la\-tion\-ship between convex, quasiconvex and star quasiconvex holds (see \cite[Proposition 5]{KL}). In the next scheme, we denote "qcx = quasiconvex":
 \begin{align}
  \begin{array}{ccccccc}
   {\rm strongly ~ convex} & \Rightarrow & {\rm strongly ~ qcx} & \overset{*}{\Rightarrow} & 
   {\rm strongly ~ star ~ qcx} \notag \\
   \Downarrow & \, & \Downarrow & \, & \Downarrow  \notag \\
   {\rm convex} & \Rightarrow & {\rm qcx} & \overset{*}{\Rightarrow} & {\rm star ~ qcx}
   \end{array}
  \end{align}
  where ``$*$" denotes that ${\rm argmin}_{ \mathbb{R}^{n}\,}h \neq \emptyset$ is required. 

  \item[$(ii)$] Interesting examples of star quasiconvex functions maybe be found in \cite{BLL,KL,NTV} and references therein, including the CES production/utility function (see \cite[Example 9]{BLL}).
  
  \item[$(iii)$] Observe that when $\gamma > 0$, relation \eqref{char:grad} shows that differentiable strongly star quasiconvex functions coincides with functions satisfying the {\it Restricted Secant Inequality} (RSI) property (see \cite{Yi,ZY} among others). 
 
  When $\gamma = 0$ in relation \eqref{char:grad}, differentiable star quasiconvex functions includes {\it Variationally Coherent} (VC) functions, which were developed in \cite{ZBoyd} (or differentiable functions which satisfies \cite[Assumption $(1.2)$]{SS}). This inclusion may be strict as \cite[Remark 28$(ii)$]{KL} shows.

  Both classes (RSI) and (VC) have been studied deeply in continuous optimization and machine learning theory in virtue of its properties for gradient-type methods.
  
  \item[$(iv)$] For further properties on (strongly) star quasiconvex functions as well as sufficient conditions for linear convergence of first-order type algorithms in optimization, we refer to \cite{KL,NTV}.
 \end{itemize}
\end{remark}

From the economical point of view, a star quasiconcave function $h$ (i.e., when $-h$  is star quasiconvex) is related to {\it star-shaped preferences}, which appears naturally in economic theory in several works (see, for instance, \cite{BJ,FOS,Mas}). 

\medskip

For a further study on generalized convexity we refer to \cite{ADSZ,CM-Book,KL,Lara-9,NTV} and refe\-ren\-ces therein.

\section{Main Results}\label{sec:03}

\subsection{Additive Separable Star Quasiconvexity}

In the following result, we characterize separable (strongly) star quasiconvex functions.

\begin{theorem}{\bf (Additive Separable Star Quasiconvexity)} \label{th:sepa}
 Let $X = \prod^{m}_{i=1} X_{i}$, $h: X \rightarrow \mathbb{R}$ be defined as in \eqref{separable}, and $\overline{x} := (\overline{x}_{1}, \ldots, \overline{x}_{m}) \in {\rm argmin}_{X}\,h$. Then, $h$ is (strongly) star quasiconvex with modulus $\gamma \geq 0$ with respect to $\overline{x}$ if and only if every $h_{i}$ is (strongly) star quasiconvex with the same modulus $\gamma \geq 0$ with respect to $\overline{x}_{i} \in {\rm argmin}_{X_{i}}\,h_{i}$.
\end{theorem}

\begin{proof}
 Clearly, $\overline{x} := (\overline{x}_{1}, \ldots, \overline{x}_{m}) \in {\rm
 argmin}_{X}\,h$ if and only if $\overline{x}_{i} \in {\rm argmin}_{X_{i}}\,h_{i}$ 
 for every $i \in \{1, \ldots, m\}$ (which holds for every separable function). 

 $(\Rightarrow)$: Assume that $h$ is (strongly) star quasiconvex with modulus $\gamma \geq 0$ with respect to $\overline{x} := (\overline{x}_{1}, \ldots,  \overline{x}_{m}) \in {\rm argmin}_{X}\,h$. Fix $i_{0} \in \{1,\ldots, m\}$. Then for every $y \in X_{i_{0}}$, consider the vector $z := (z_{1}, \ldots, z_{m})$ such that $z_{j} = \overline{x}_{j}$ for all $j \neq i_{0}$ and $z_{i_{0}} = y$. Then, for every $\lambda \in [0, 1]$, we have
 \begin{equation}\label{h}
  h (\lambda \overline{x} + (1-\lambda) z) \leq h(z) - \lambda (1-\lambda) \frac{\gamma}{2} \lVert z - \overline{x} \rVert^{2}. 
 \end{equation}

 Since $h_{j} (\lambda \overline{x} + (1-\lambda) z) = h_{j}(\overline{x}_{j})$ for every $j \neq i_{0}$, and $h (z) = \sum_{j \neq i_{0}} h_{j} (\overline{x}_{j}) + h_{i_{0}} (y)$,  re\-la\-tion \eqref{h} gives
\begin{align*}
 h_{i_{0}} (\lambda \overline{x}_{i_{0}} + (1-\lambda) y) & \leq h_{i_{0}} (y) - \lambda (1-\lambda) \frac{\gamma}{2} \lVert y - \overline{x}_{i_{0}} \rVert^{2}, 
\end{align*}
 thus $h_{i_{0}}$ is (strongly) star quasiconvex with mudulus $\gamma \geq 0$ with respect to $\overline{x}_{i_{0}} \in {\rm argmin}_{X_{i_{0}}}\,h_{i_{0}}$. Since $i_{0}$ was ar\-bi\-tra\-ry, every $h_{i}$ is (strongly) star quasiconvex with mudulus $\gamma \geq 0$ with respect to $\overline{x}_{i}$. 

 $(\Leftarrow)$: Let $\widehat{y} := (y_{1}, \ldots, y_{m}) \in X$. Since every $h_{i}$ is (strongly) star quasiconvex with modulus $\gamma \geq 0$  with respect to $\overline{x}_{i}$, we have
 \begin{align*}
  h_{i} (\lambda \overline{x}_{i} + (1-\lambda)y_{i}) \leq h_{i} (y_{i}) - \lambda
  (1-\lambda) \frac{\gamma}{2} \lVert y_{i} - \overline{x}_{i} \rVert^{2}, ~ 
  \forall ~ i \in \{1, \ldots, m\}.
 \end{align*}
 Summing up from $i=1$ to $i=m$, and since $h \sum_{i=1}^{m} h_{i}$, we obtain
 \begin{align*}
  h(\lambda \overline{x} + (1-\lambda) \widehat{y}) & \leq \sum^{m}_{i=1} h_{i} (y_{i}) - \lambda (1-\lambda) \frac{\gamma}{2}
 \sum^{m}_{i=1} \lVert y_{i} - \overline{x}_{i} \rVert^{2} \\
 & = h(\widehat{y}) - \lambda (1-\lambda) \frac{\gamma}{2} \lVert \widehat{y} 
  - \overline{x} \rVert^{2},
 \end{align*}
 i.e., $h$ is (strongly) star quasiconvex with modulus $\gamma \geq 0$ with respect to $\overline{x}$. 
\end{proof}

As a consequence, we solves problem \eqref{DK} when ${\rm argmin}_{X}\,h \neq \emptyset$.

\begin{corollary}\label{DK:problem} {\bf (Solving Debreu-Koopmans Problem)}
 Let $X = \prod^{m}_{i=1} X_{i}$, $h: X \rightarrow \mathbb{R}$ be defined as in \eqref{separable}, and $\overline{x} = (\overline{x}_{1}, \ldots, \overline{x}_{m}) \in {\rm argmin}_{X}\,h$. If every $h_{i}$ is (strongly) quasiconvex with modulus $\gamma \geq 0$, then $h$ is (strongly) star quasiconvex with the same modulus $\gamma \geq 0$ with respect to $\overline{x}$.
\end{corollary}

\begin{proof}
 Since every $h_{i}$ is (strongly) quasiconvex with modulus $\gamma \geq 0$ and $\overline{x}_{i} \in {\rm argmin}_{X_{i}}\,h_{i}$ for every $i \in \{1, \ldots, m\}$, $h_{i}$ is (strongly) star quasiconvex with mo\-du\-lus $\gamma \geq 0$ with respect to $\overline{x}_{i}$ for every $i \in \{1, \ldots, m\}$ by Remark \ref{rem:relat}$(i)$, and the result follows by Theorem \ref{th:sepa}.
\end{proof}

In the following example, we apply the previous result to classical economic theory in social welfare (see \cite{Berg,Samuel}).

\begin{example}\label{ex:berg-samuel}
{\bf (Bergson-Samuelson Social Welfare)}
Let $x_i \in X_i$ be the consumption vectors for each individual $i = 1, \ldots, m$, where $X_i \subseteq \mathbb{R}^{n_{i}}$ is a convex set and  $\sum^{m}_{i=1} n_{i} = n$. Let $u_i: X_i \to \mathbb{R}$ be the utility function of each consumer, which represents their individual preferences. Then the Bergson-Samuelson Social Welfare Function (SWF) is defined (in its additively decomposable form) as $W: X \to \mathbb{R}$ given by
\begin{equation}\label{SWF}
 W(x_1, \ldots, x_m) = \sum_{i=1}^{m} u_i(x_i).
\end{equation} 
   
The Debreu-Koopmans theorem implies that for a separable welfare function $W$ to be quasiconvex, at least $m-1$ of the individual utility functions must be convex. This imposes severe restrictions on admissible preference patterns.

On the other hand, if each function $u_i$ is quasiconvex, then $W$ is star quasiconvex by Corollary \ref{DK:problem}, provided that each $u_i$ has a minimizer on $X_i$. 

Since quasiconvex functions are related to convex preferences, while star quasiconvex functions are related to star-shaped preferences, we conclude that {\rm the aggregation of individual convex preferences yields a star-shaped social pre\-fe\-ren\-ce relation.}
\end{example}

Another consequence of Corollary \ref{DK:problem} is that separable value functions in Prospect Theory (see \cite{KT,TK,Wakker}) are (strongly) star quasiconvex when the sets $X_{i} \subseteq \mathbb{R}$. Indeed, as noted in the seminal work of Kahneman and Tversky \cite[page 279]{KT}, the value functions in Prospect Theory are $S$-shaped (more generally, a sigmoid-type function): concave for gains and convex for losses, reflecting risk aversion for gains and risk seeking for losses. In one dimension, such $S$-shaped functions are nondecreasing, and every nondecreasing function is quasiconvex (see \cite[Theorem 2.5.1]{CM-Book}), thus each component of a separable $S$-shaped value function is quasiconvex. Consequently, by applying Corollary \ref{DK:problem}, the aggregate separable $S$-shaped function is star quasiconvex, provided it attains a minimizer.

We have proved the following result:

\begin{corollary}\label{PTVF:separable} {\bf (Star Quasiconvexity of Separable Prospect Theory Va\-lue/Utility Functions)}
 Let $X_{i} \subseteq \mathbb{R}$ be convex sets, $h_{i}: X_{i} \rightarrow \mathbb{R}$ be functions, $\overline{x}_{i} \in {\rm argmin}_{X_{i}}\,h_{i}$, and $h: X \rightarrow \mathbb{R}$ be such that relation \eqref{separable} holds. If every $h_{i}$ has a {\it S-shape} form (that is, concave for gains and convex for losses), then the aggregate function $h$ is star quasiconvex on $X$ with respect to $\overline{x} = (\overline{x}_{1}, \ldots, \overline{x}_{n}) \in {\rm argmin}_{X}\,h$.
\end{corollary}

As an example, let us consider the following  piecewise power value/utility function with probability weighting (see, for instance, \cite{Wak-1}). 
\begin{example} {\bf (Prospect Theory Value/Utility Functions with Probability Weighting)} \label{ex:PT}
 Let $X_i = [a_{i}, b_{i}] \subseteq \mathbb{R}$ be compact convex sets for every $i \in \{1, \ldots, n\}$, with $a_{i} < b_{i}$ for every $i$. Furthermore, every prospect has:
\begin{itemize}
    \item[$(i)$] Outcome $x_i \in X_i$.
    \item[$(ii)$] Associated probability $p_i \in \, ]0,1]$ with $\sum_{i=1}^{n} p_i = 1$.
    \item[$(iii)$] Decision weight $\pi_i = w(p_i)$ where $w:[0,1]\to[0,1]$ is a probability weighting function.
\end{itemize}
Define the value/utility function $u_i: X_i \to \mathbb{R}$ by 
\begin{equation}\label{W:example}
u_i(x_i) = 
\begin{cases}
 ~~ x_i^{\alpha_i} & {\rm for ~} x_i \geq 0 \quad {\rm (gains)}, \\[8pt]
-\lambda_i (-x_i)^{\beta_i} & {\rm for ~} x_i < 0 \quad {\rm (losses)},
\end{cases}
\end{equation}
where $\alpha_i, \beta_i \in \, ]0, 1[$ and $\lambda_i > 1$ for every $i \in \{1, \ldots, n\}$.

Then the Separable Prospect Theory valuation (see \cite{Wak-3}) for the bundle of prospects is 
\begin{equation}\label{pvalue:func}
 V(x_1, \ldots, x_n) = \sum_{i=1}^{n} \pi_i u_i(x_i) = \sum_{i=1}^{n} h_i(x_i), 
\end{equation}
where $h_i(x_i) = \pi_i  u_i(x_i)$.

Note that function $V$ is star quasiconvex (with modulus $\gamma = 0$) with respect to its minimizer $\overline{x} = (a_1, \ldots, a_n) \in X$. Indeed, every $u_{i}$ is nondecreasing on $X_{i}$, and since $p_{i} > 0$ and $\pi_{i} > 0$, $h_{i}$ is nondecreasing on $X_{i}$, too, thus quasiconvex on $X_{i}$ by \cite[Theorem 2.5.1]{CM-Book}. Therefore, by Corollary \ref{DK:problem}, the aggregate Prospect Theory valuation function $V$ inherits star quasiconvexity from its components, with minimizer $\overline{x} = (a_1,\ldots,a_n)$.

An illustration of function $V$ and its star-shaped sublevel sets are given in Figures \ref{fig:01} and \ref{fig:02}, respectively.

Moreover, if every $u_i$ is (strongly) star quasiconvex with mo\-du\-lus $\gamma_i \geq 0$ with respect to its minimizer $a_i \in X_{i}$, then every $h_i$ is (strongly) star quasiconvex with modulus $\pi_i\gamma_i \geq 0$ and, consequently, $V$ is (strongly) star quasiconvex with respect to $\overline{x} = (a_1, \ldots, a_{n})$ with modulus $\gamma = \min_{i} \{\pi_i \gamma_i\} \geq 0$.

\begin{figure}[htbp] 
\centering \includegraphics[width=1.00\linewidth]{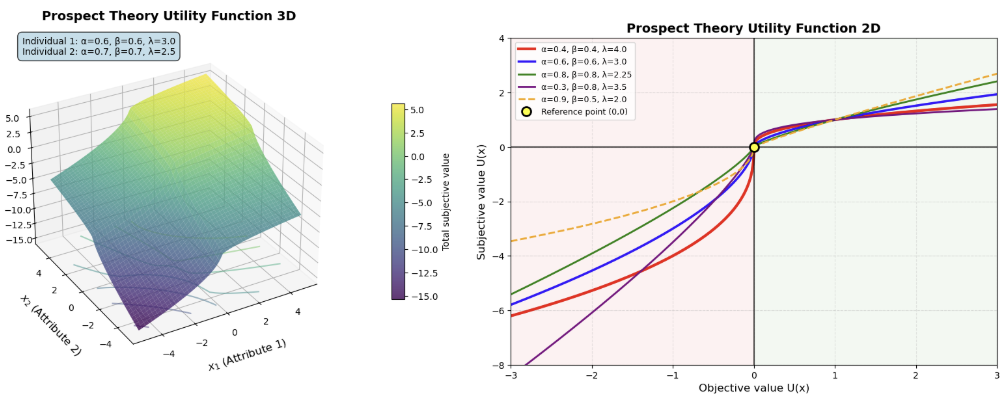} 
\caption{Function $V$ in \eqref{pvalue:func}. A 3D plot of $V$ (left), and 2D plots for different parameters (right).} \label{fig:01}
\end{figure}
\begin{figure}[htbp] 
\centering \includegraphics[width=1.0\linewidth]{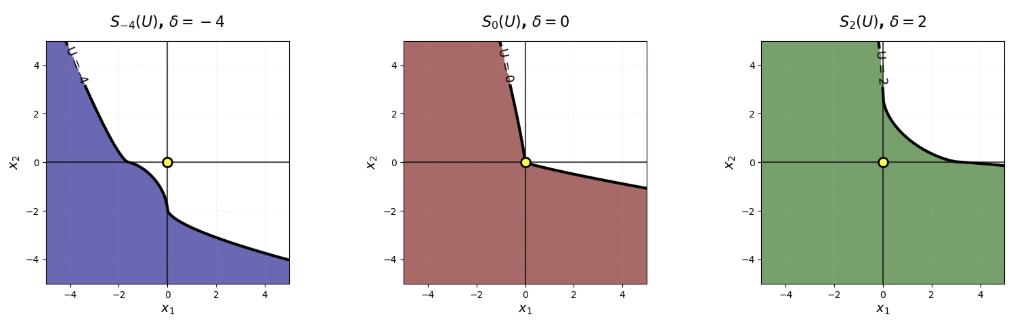} 
\caption{An illustration of the sublevel sets at height $\delta = -4, 0, 2$ of function $V$ des\-cri\-bed in \eqref{pvalue:func}. All sublevel sets are star shaped w.r.t. $(\overline{x}_{1}, \overline{x}_{2}) = (-5, -5)$.} \label{fig:02}
\end{figure}
\end{example}

Formally, we can prove the following result on compact convex sets. Its proof follows the argument of Example \ref{ex:PT}, so it is omitted.

\begin{corollary}{\bf (Prospect Theory with Probability Weighting)} \label{cor:PTweights}
Let $X = \prod_{i=1}^{n} X_i$ with every $X_i = [a_{i}, b_{i}] \subseteq \mathbb{R}$ be compact convex sets with $a_{i} < b_{i}$ for every $i \in \{1, \ldots, n\}$. Let $u_i: X_i \to \mathbb{R}$ be defined as in \eqref{W:example}, and $\pi_i > 0$ be decision weights derived from a probability weighting function $w: [0,1] \to [0,1]$. Define the functions 
$$h_i(x_i) = \pi_i u_i(x_i) ~~ {\rm and} ~~ V(x) = \sum_{i=1}^{n} h_i (x_i).$$ 
If every $u_i$ is (strongly) star quasiconvex with modulus $\gamma_i \geq 0$ with respect to $a_i \in {\rm argmin}_{X_i} u_i$, then the following assertions hold:
\begin{enumerate}
 \item[$(a)$] Every $h_i$ is (strongly) star quasiconvex with modulus $\pi_i \gamma_i \geq 0$ with respect to $a_i$.

 \item[$(b)$] $V$ is (strongly) star quasiconvex with modulus $\gamma = \min_{i} \{\pi_i \gamma_i\}$ with respect to $\overline{x} = (a_1, \ldots, a_n)$.
\end{enumerate}
\end{corollary}

The following remark shows that the components of a star quasiconvex function may not be necessarily continuous.

\begin{remark}
In \cite[Theorem 1]{DK}, the authors proved that if $h = \sum_{i} h_{i}$ is 
quasiconvex, then all components $h_{i}$ are continuous. This result cannot be 
extended to star quasiconvex functions, as the following example shows. Let $h: \mathbb{R}^{2} \rightarrow \mathbb{R}$ be given by $h(x_1, x_2) =
 h_1(x_1) + h_2(x_2)$, where 
\[
h_1(t) =
\begin{cases}
0 & {\rm if } ~ t = 0, \\
1 & {\rm if}  ~ t \neq 0,
\end{cases}
\quad {\rm and} \quad 
h_2(t) = t^2.
\]
Clearly, $h_{1}$ is quasiconvex, while $h_{2}$ is convex. Hence, it satisfies the necessary condition of the Debreu-Koopmans theorem (at maximum one nonconvex component). However, $h$ is not quasiconvex (take $x = (0, 1)$ and $y=(0.1, 0)$, then $h(x)=1$, $h(y) = 1 + \varepsilon^2$, but $h(\frac12 x + \frac12 y) = 1 + \frac14 > 1$, violating quasiconvexity).

On the other hand, $h$ is star quasiconvex with respect to $\overline{x} = (0, 0)$ by Corollary \ref{DK:problem}, even when $h_1$ is discontinuous at $0$. 
\end{remark}

\subsection{Monotonic Transformations}

For the next result, and in order to avoid a differentiability assumption, we will use tools from nonconvex optimization and nonsmooth analysis. To that end, let $g: \, ]a, b[ \, \rightarrow \mathbb{R}$ be a function. Then the lower Dini directional derivative of $g$ at $x \in \, ]a, b[$ is defined as
\begin{equation*}\label{lower:dini}
 g^{D-} (x) = \liminf_{t \rightarrow 0^{+}} \frac{g(x+t) - g(x)}{t},
\end{equation*}
where $t \rightarrow 0^{+}$ means $t \to 0$ with $t>0$.

The following mean value theorem holds.

\begin{theorem}\label{MVT} {\rm (\cite[Theorem I.3.1]{DeRu})}
 Let $g: I \, \rightarrow \mathbb{R}$ be a continuous function on some open interval $I \subseteq \mathbb{R}$. Let $a, b \in I$ and $\mu := \inf_{t \in [a, b]} g^{D-} (t)$. Then, 
 $$g(b)-g(a) \geq \mu (b-a).$$
\end{theorem}

If the function $g$ in the previous theorem is nondecreasing, then obviously
$g^{D-} (t) \geq 0$ for all $t \in [a, b]$, so $m \geq 0$.

\medskip

We also consider the following straightforward fact:
\begin{itemize}
 \item[{\rm {\bf Fact 1}}] Let $g$ be an increasing function. Then, $\overline{x} \in {\rm argmin}_K h$ if and only if $\overline{x} \in {\rm argmin}_K (g \circ h)$.
\end{itemize} 

\begin{theorem}\label{prop:comp} {\bf (Monotonic Composition Theorem)}
 Let $K \subseteq \mathbb{R}^{n}$ be an open convex set, $h: K \rightarrow \mathbb{R}$ be a (strongly) star quasiconvex function with modulus $\gamma \geq 0$ with respect to $\overline{x} \in {\rm argmin}_{K}\,h$, and $g$ be a increasing continuous function. Let $\mu := \inf_{t \in h(K)} g^{D-} (t)$. Then the function $g \circ h$ is (strongly) star quasiconvex with modulus $\mu \gamma \geq 0$ with respect to $\overline{x} \in {\rm argmin}_{K}(g \circ h)$.
\end{theorem}

\begin{proof}
 Since $g$ is increasing, $\overline{x} \in {\rm argmin}_{K}\,h$ if and only if $\overline{x} \in {\rm argmin}_{K} (g \circ h)$ (by {\bf Fact 1}). Hence, for every $x \in K$ and every $\lambda \in [0, 1]$, we apply Theorem \ref{MVT} to $a := h(\lambda \overline{x} 
 + (1-\lambda)x)$ and $b:=h(x)$, thus
\begin{align*}
 g(h(x)) - g(h(\lambda \overline{x} + (1-\lambda) x)) & \geq \mu (h(x) - h(\lambda \overline{x} + (1 - \lambda) x)) \\
 & \geq \mu \lambda (1 - \lambda) \frac{\gamma}{2} \lVert x - \overline{x} 
 \rVert^{2},
\end{align*}
and the result follows.
\end{proof}

From this result, we can obtain the following.

\begin{corollary}\label{inverse}
 Let $K \subseteq \mathbb{R}^{n}$ be an open convex set and $h: K \rightarrow \mathbb{R}$ be a function. If $h$ is (strongly) star quasiconcave  with modulus 
 $\gamma \geq 0$ and either positive or negative on $K$, then $\frac{1}{h}$ is (strongly) star quasiconvex with the same modulus.
\end{corollary}

\begin{proof}
 If $h$ is positive, then we consider $g: ]0, + \infty[ \, \to \, ]0, + \infty[$, and if $h$ is negative, we consider $g: ]- \infty, 0[ \, \to \, ]- \infty, 0[$, in both cases, we take $g(x) = - \frac{1}{x}$.

 Now, observe that $\frac{1}{h} = - g \circ h$. Since $g^{\prime} (x) = \frac{1}{x^{2}} > 0$ for all $x \neq 0$, $g$ is increasing over $h(K)$. Hence, by Theorem  \ref{prop:comp}, $g \circ h = - \frac{1}{h}$ is (strongly) star quasiconvex with modulus $\mu \gamma \geq 0$, where $\mu = \inf_{t \in h(K)} g^{D-} (t)$. Therefore, $\frac{1}{h}$ is (strongly) star quasiconcave with modulus $\mu \gamma \geq 0$.
\end{proof}

\begin{remark}
 \begin{itemize} 
  \item[$(i)$] If $g$ is differentiable and increasing, then $\mu \geq 0$ is the infimum of its derivative.

  \item[$(ii)$] The case $\mu \gamma=0$ is possible. This can happen if $h$ is star quasiconvex ($\gamma=0$), or if $\mu=0$, which may happen if $g$ has flat parts or, for instance, if $g(t) = \ln t$, $t \in \, ] 0, + \infty[$. If $h$ is defined on a bounded set $K$, $h(K)$ might be
  bounded, so for $g(t) = \ln t$, we have $\mu > 0$. Then, on bounded sets, $g \circ h$ is strongly star quasiconvex if $h$ is strongly star quasiconvex ($\gamma > 0$).

 \item[$(iii)$] Corollary \ref{inverse} says that the multiplicative inverse of a (strongly) star quasiconvex function is (strongly) star quasiconcave, as well as the inverse of a (strongly) star quasiconcave function is (strongly) star quasiconvex. Note that convex/concave functions does not have this relationship with their inverse as the function $h(x) = e^{x}$ shows.
 \end{itemize}
\end{remark}

Another interesting calculus rule is given below. We emphasize that the following result holds for $\gamma = 0$.

\begin{proposition}\label{prop:leo}
 Let $X = \prod^{n}_{i=1} X_{i}$, where $X_{i} = [a_{i}, b_{i}]$ with $0 < a_{i} < b_{i}$ for every $i \in \{1, \ldots, n\}$, and $h: X \rightarrow \mathbb{R}$ be given by 
 \begin{equation}\label{min:sep}
  h (x_{1}, \ldots, x_{n}) = \min_{i=1,\dots,n} h_i(x_i),
 \end{equation} 
 where every $h_i: X_{i} \to \mathbb{R}$ is a star quasiconvex function with respect to $a_i \in {\rm argmin}_{ X_{i}} \,h_{i}$. Then $h$ is star quasiconvex with respect to $\overline{x} = (a_1, \dots, a_n)$.
\end{proposition}

\begin{proof}
 Since every $h_i$ is star quasiconvex with respect to $a_i \in {\rm argmin }_{X_{i}}\,h_{i}$, we have for every $y_i \in X_{i}$ and every $\lambda \in [0,1]$ that $h_i (\lambda a_i + (1-\lambda) y_i) \le h_i (y_i)$. 
 Fix $y = (y_1,\dots,y_n)$ and let $i_0$ be an index where the minimum is attained, $h(y) = \min_i h_i(y_i) = h_{i_0} (y_{i_0})$. Then,
\begin{align*}
 h(\lambda \overline{x} + (1-\lambda)y) 
 = \min_i h_i (\lambda a_i + (1-\lambda) y_i) & \leq h_{i_0} (\lambda a_{i_0} + (1-\lambda) y_{i_0}) \\
 & \le h_{i_0} (y_{i_0}) = h(y).
\end{align*}
Hence $h$ is star quasiconvex with respect to $\overline{x} = (a_1, \dots, a_n)$.
\end{proof}

In particular, the Leontief \cite{Leon} production/utility function (perfect complements) is star quasiconvex.

\begin{example}\label{Leon:func}
{\bf (Star Quasiconvexity of Leontief Function)} 
Let $\alpha > 0$ and $\alpha_{i} > 0$ for all $i \in \{1, \ldots, n\}$. 
Then the Leontief (see \cite{Leon}) production/utility function  $h_{L}: \mathbb{R}^{n}_{+} \rightarrow \mathbb{R}_{+}$ is given by 
\begin{equation}\label{Leontief}
  h_{L}(x) = \min\left\{ \frac{1}{\alpha_{1}} x_{1}, \ldots, \frac{1}{\alpha_{n}} x_{n} \right\}^{\alpha}.
\end{equation}
It is known that $h_{L}$ is quasiconcave, and it is concave if and only if 
$0 < \alpha \leq 1$ (see, for instance, \cite[Theorem 2.4.3]{CM-Book}).

Note that on any convex bounded domain $X = \prod_{i=1}^n [a_i, b_i]$ with $0 < a_i < b_i < + \infty$ for every $i \in \{1, \ldots, n\}$, the function $h_L$ attains its minimum at $\overline{x} = (a_1, \dots, a_n)$. 

We can write $h_L(x) = g(h(x))$ where
\[
 h(x) = \min\{ h_1(x_1), \dots, h_n(x_n) \},\quad 
h_i(x_i) = \frac{1}{\alpha_i} x_i,\quad 
g(t) = t^\alpha.
\]
Since every $h_i$ is linear (hence star quasiconvex with respect to $a_i$ with modulus $0$), we have by Proposition \ref{prop:leo} that $h$ is star quasiconvex with respect to $\overline{x} = (a_1, \dots, a_n)$. Furthermore, since $g$ is increasing and continuous for every $\alpha > 0$ on $\mathbb{R}_{+}$, using Theorem \ref{prop:comp} we conclude that $h_L = g \circ h$ is star quasiconvex with respect to $\overline{x}$ on any convex bounded domain (for every $\alpha > 0$).
\end{example}


\subsection{Product Separable Star Quasiconvexity}

Using Theorems \ref{th:sepa} and \ref{prop:comp}, we provide sufficient conditions for star quasiconvexity of product separable star quasiconvex functions.

\begin{theorem} {\bf (Product Separable Star Quasiconvexity)} \label{th:separable-product}
Let $X = \prod_{i=1}^m X_i$, where each $X_i \subseteq \mathbb{R}^{n_i}$ is a convex set with $n_i \geq 1$ and $\sum^{m}_{i} n_{i} = n$, and $H: X \rightarrow \mathbb{R}$ be defined by
\begin{equation}\label{sep:product-gen}
 H (x_1,\dots,x_m) = \prod_{i=1}^m h_i(x_i),
\end{equation}
where $h_i: X_i \rightarrow \mathbb{R}$ satisfy $h_i(x_i) > 0$ for all $x_i \in X_i$ and all $i = 1,\dots,m$. Let $\overline{x} = (\overline{x}_1,\dots,\overline{x}_m) \in {\rm  argmin}_{X}\,H$. Then the following assertions hold:
\begin{itemize}
 \item[$(a)$] $H$ is star quasiconvex with respect to $\overline{x}$ if and only if every $h_i$ is star quasiconvex with respect to $\overline{x}_{i}$.

 \item[$(b)$] If every function $f_i := \ln h_i$ is (strongly) star quasiconvex with modulus $\gamma \geq 0$ with respect to $\overline{x}_{i}$, then $H$ is (strongly) star quasiconvex with mo\-du\-lus $\gamma_{H} = \gamma H(\overline{x})$ with respect to $\overline{x}$.
    
 \item[$(c)$] If every $h_i$ is (strongly) star quasiconvex with modulus $\gamma_i \geq 0$ with respect to $\overline{x}_{i}$, and there exists $M_{i} > 0$ such that $h_i(y_i) \leq M_i$ for all $y_i \in X_i$, then $H$ is (strongly) star quasiconvex with respect to $\overline{x}$ with modulus 
 $$\gamma_{H} \geq H(\overline{x}) \left( \min_{1 \leq i \leq m} \frac{\gamma_i}{M_i} \right).$$
\end{itemize}
\end{theorem}

\begin{proof}
Since $h_i > 0$, the function $F(x) := \ln H(x) = \sum_{i=1}^m \ln h_i(x_i)$ is well-defined. As $\ln (\cdot)$ is increasing, $\overline{x}$ minimizes $H$ if and only if it minimizes $F$ by {\bf Fact 1}. Hence, for the additive function $F$, minimization occurs componentwise: $\overline{x}$ minimizes $F$ 
if and only if every $\overline{x}_i$ minimizes $\ln h_i$. Again, since $\ln (\cdot)$ is increasing, this is equivalent to $\overline{x}_i$ minimizing $h_i$. Therefore,  $\overline{x}_i \in {\rm argmin}_{X_i} h_i$ for all $i \in \{1, \ldots, m\}$.

$(a)$: $(\Leftarrow)$ Suppose that every $h_i$ is star quasiconvex. Then for every $y_{i} \in X_{i}$ and every $\lambda \in [0,1]$, star quasiconvexity of $h_i$ implies
\[
h_i(\lambda \overline{x}_i + (1-\lambda)y_i) \leq h_i(y_i).
\]
Since $h_i > 0$ for all $i$, multiplying these inequalities for every $i = 1,\dots,m$ and taking $y = (y_1,\dots,y_m) \in X$ yields
\[
 H(\lambda \overline{x} + (1-\lambda)y) = \prod_{i=1}^m h_i(\lambda \overline{x}_i + (1-\lambda)y_i) \leq \prod_{i=1}^m h_i(y_i) = H(y).
\]
Thus $H$ is star quasiconvex with respect to $\overline{x}$.

$(\Rightarrow)$ Suppose that $H$ is star quasiconvex. Fix $i_0 \in \{1,\dots,m\}$ and take $y \in X_{i_0}$. Define $z = (z_1,\dots,z_m) \in X$ by $z_j = \overline{x}_j$ for $j \neq i_0$ and $z_{i_0} = y$. Then for every $\lambda \in [0,1]$, we have
\begin{eqnarray*}
 H (\lambda \overline{x} + (1-\lambda)z) & \leq & H(z) \\ 
 \Longleftrightarrow \left( \prod_{j \neq i_0} h_j(\overline{x}_j) \right) h_{i_0}(\lambda \overline{x}_{i_0} + (1-\lambda)y) & \leq & \left( \prod_{j \neq i_0} h_j(\overline{x}_j) \right) h_{i_0}(y).
\end{eqnarray*}
Since $\prod_{j \neq i_0} h_j(\overline{x}_j) > 0$, we obtain $h_{i_0}(\lambda \overline{x}_{i_0} + (1-\lambda)y) \leq h_{i_0}(y)$. Hence, $h_{i_0}$ is star quasiconvex with respect to $\overline{x}_{ i_{0}}$. Since $i_0$ was arbitrary, every $h_i$ is star quasiconvex with respect to $\overline{x}_{i}$.

$(b)$: For the function $F = \ln H$, we use Theorem \ref{prop:comp} with $g(t) = e^t$. Since $g$ is increasing and differentiable with $g'(t) = e^t$, we have $g^{D-}(t) = e^t$. Furthermore, the infimum of $g^{D-}$ on $F(X)$ is
$$m = \inf_{t \in F(X)} e^t = \exp\left(\inf F(X) \right) = \exp(F(\overline{x})) = H (\overline{x}) > 0,$$
because $h_{i} > 0$ on $X_{i}$ for all $i \in \{1, \ldots, m\}$ by assumption.

Since every $f_i = \ln h_i$ is (strongly) star quasiconvex with respect to $\overline{x}_{i}$ with modulus $\gamma \geq 0$, it follows from Theorem \ref{th:sepa} that $F$ is (strongly) star quasiconvex with the same modulus $\gamma \geq 0$, and by Theorem \ref{prop:comp}, $H = g \circ F$ is (strongly) star quasiconvex with respect to $\overline{x}$ with modulus
$$ m \gamma = H(\overline{x}) \gamma \geq 0.$$

$(c)$: Suppose that every $h_i$ is (strongly) star quasiconvex with modulus $\gamma_i \geq 0$ with respect to $\overline{x}_{i}$ and $h_i(y_i) \leq M_i$ for all $y_i \in X_i$. Then for every $\lambda \in [0,1]$ and every $y_i \in X_i$,
\begin{align}
 & \hspace{0.9cm} h_i(\lambda \overline{x}_i + (1-\lambda)y_i) \leq h_i(y_i) - \lambda (1-\lambda) \frac{\gamma_i}{2}\|y_i-\overline{x}_i\|^2 \notag \\
 & \Longrightarrow ~ 0 < \frac{h_i(\lambda \overline{x}_i + (1-\lambda)y_i)}{h_{i} (y_{i})} \leq 1 - \lambda (1-\lambda) \frac{\gamma_i}{2 h_{i} (y_{i})}\|y_i-\overline{x}_i\|^2, \label{for:c}
\end{align}
because $h_i > 0$. 

Set $t := \frac{\lambda(1-\lambda)\gamma_i\|y_i-\overline{x}_i\|^2}{2h_i(y_i)}$. Then, it follows from \eqref{for:c} that $0 < t < 1$. Taking logarithms on both sides of \eqref{for:c}, 
$$\ln \left( \frac{h_i(\lambda \overline{x}_i + (1-\lambda)y_i)}{h_{i} (y_{i})} \right) \leq \ln(1-t),$$
and using $\ln(1-t) \leq -t$, we obtain 
$$ \ln h_{i} (\lambda \overline{x}_i + (1-\lambda)y_i) \leq \ln h_{i} (y_i) - \lambda(1-\lambda)\frac{\gamma_i}{2h_i(y_i)}\|y_i-\overline{x}_i\|^2.$$
Since $0 < h_i(y_i) \leq M_i$, we have $-\frac{\gamma_i}{h_i(y_i)} \leq - \frac{\gamma_i}{M_i}$, thus
 $$\ln h_{i} (\lambda \overline{x}_i + (1-\lambda) y_i) \leq \ln h_{i} (y_i) - \lambda (1 - \lambda) \frac{\gamma_i}{2M_i} \|y_i - \overline{x}_i\|^2.$$
 Then, every $\ln h_i$ is (strongly) star quasiconvex with respect to $\overline{x}_{i}$ with modulus at least $\gamma_i/M_i$. In particular, they are all (strongly) star quasiconvex with the same modulus $\gamma := \min_i (\gamma_i/M_i)$. Applying part $(b)$ (of this theorem) with $\gamma = \min_i \frac{\gamma_i}{M_i}$, we conclude that $H$ is (strongly) star quasiconvex with respect to $\overline{x}$ with modulus
 $$\gamma_{H} = H (\overline{x}) \gamma = H (\overline{x}) \left( \min_{1 \leq i \leq m} \frac{\gamma_i}{M_i} \right) \geq 0.$$
This completes the proof.
\end{proof}

\begin{remark}
 \begin{itemize}
  \item[$(i)$] The reverse statement in part $(b)$ does not hold in general. Indeed, take $m = 1$, $X = [1, 3]$, $\overline{x} = 1$, and $H(x) = x$. Then $h$ is strongly star quasiconvex with modulus $\gamma_h = 1$ with respect to $\overline{x} = 1$. In fact, for every $y \in [1,3]$ and every $\lambda \in [0,1]$, we have
 \begin{align*}
  & H(\lambda \overline{x} + (1-\lambda) y) = \lambda + (1-\lambda) y = y - \lambda(y-1), \\ 
  & H(y) - \lambda (1-\lambda) \frac{1}{2} (y-1)^2 = y - \lambda (1-\lambda) \frac{1}{2} (y-1)^2.
 \end{align*}
  Hence, $y - \lambda (y-1) \leq y - \lambda (1-\lambda) \frac{1}{2} (y-1)^2$ reduces to $1 \geq (1-\lambda) \frac{(y-1)}{2}$, which holds for all $y \in [1, 3]$ and all $\lambda \in [0, 1]$ with $\gamma_{H} = 1$.
    
  On the other hand, $F(x) = \ln H(x) = \ln x$ is not strongly star quasiconvex with modulus $\gamma_{H} / H (\overline{x}) = 1$. In fact, take $y = 3$ and $\lambda = 0.5$, thus
  \begin{align*}
   & F(\lambda \overline{x} + (1-\lambda) y) = \ln(2) \approx 0.69, \\
   & F(y) - \lambda(1-\lambda) \frac{1}{2} (y-1)^2 = \ln(3) - 0.25 \cdot 2 \approx 0.59.
  \end{align*}
  Since $0.69 > 0.59$, $F$ is not strongly star quasiconvex with modulus $1$.
    
  \item[$(ii)$] The boundedness assumption $h_i(y_i) \leq M_i$ in part $(c)$ is common in economics, and it holds immediately when domains $X_i$ are bounded (e.g., wealth or consumption bounded above). If $X_i$ is unbounded, one may still obtain a finite modulus $\gamma_{H}$; in this case, $M_i$ should be replaced by $\sup_{y_i \in X_i} h_i(y_i)$ (provided this supremum is finite).
 \end{itemize} 
\end{remark}

The following results shows that the Cobb-Douglas function \cite{CD} is star quasiconvex on compacts convex sets.

\begin{corollary} {\bf (Star Quasiconvexity of  Cobb-Douglas Function)}  \label{cor:cobb-douglas}
 Let $H: \mathbb{R}^n_+ \to \mathbb{R}$ be the Cobb-Douglas function \cite{CD} defined by
 \begin{equation}\label{cd:funct}
  H(x_1,\ldots,x_m) = A \prod_{i=1}^m x_i^{\alpha_i},
 \end{equation}
 where $A > 0$, $\alpha_i > 0$, and $m \ge 1$. Let $X = \prod_{i=1}^m X_i$ with $X_i = [a_i, b_i]$ where $0 < a_i < b_i < +\infty$. Then  $H$ is star quasiconvex on $X$ with respect to its mi\-ni\-mi\-zer $\overline{x} = (a_1,\ldots,a_m)$.
\end{corollary}

\begin{proof}
 Since every $x_{i} \mapsto x_i^{\alpha_i}$, $\alpha_{i} > 0$, is increasing on $[a_i, b_i]$, the Cobb-Douglas function $H(x) = A \prod_{i=1}^m x_i^{\alpha_i}$ attains its minimum at $\overline{x} = (a_1, \ldots, a_m)$.

 Denote $H(x) = \prod_{i=1}^m h_i(x_i)$ with $h_i(x_i) = A_i x_i^{\alpha_i}$ and $\prod_{i=1}^m A_i = A$. Since every $h_i$ is increasing on $[a_i, b_i]$, we have for every $y_i \in X_i$ and every $\lambda \in [0,1]$ that $\lambda a_i + (1-\lambda)y_i \leq y_{i}$ and thus
 $$h_i(\lambda a_i + (1-\lambda)y_i) \le h_i(y_i),$$
 then $h_i$ is star quasiconvex with respect to $a_i$. Therefore, $H$ is star quasiconvex with respect to $\overline{x} = (a_1,\ldots,a_m)$ by Theorem \ref{th:separable-product}$(a)$.
\end{proof}

\begin{remark}
 Since functions $x_{i} \mapsto (x_{i})^{\alpha_{i} }$, $\alpha_{i} > 0$, are increasing, their minimum is attained at the left-hand side of $X_{i}$, hence Corollary \ref{cor:cobb-douglas} can be easily extended to the case of unbounded sets of the type $X_{i} = [a_{i}, + \infty[$.
\end{remark}

\section{Applications}\label{sec:04}

\subsection{Weighted Quasi-Arithmetic Means}\label{subsec:4-1}

Let $f:\mathbb{R} \rightarrow \mathbb{R}$ be a continuous and strictly monotonic function; hereafter it will be called the {\it mean generator}.
Then we define the {\it weighted quasi-arithmetic mean} (WQAM henceforth) as follows:
\begin{equation}\label{wqam}
 M_{f} (x) = f^{-1}\left( \sum_{i=1}^{n} w_i f(x_i) \right), \tag{WQAM}
\end{equation}
where the weights $w_i >0$ for every $i=1,\ldots,n$  and $\sum_{i=1}^{n} w_i=1$. 

Let us note that quasi-arithmetic means have been studied since almost one century ago \cite{knopp,kolmogorov}, and its importance in all research fields is outstanding. 

\medskip

Before continuing, we recall that given a function $g: \mathbb{R}^{n} \rightarrow \mathbb{R}$ and a set $K \subseteq \mathbb{R}^{n}$, then $g|_{K}$ denotes the restriction of the function $g$ to the set $K$.

\medskip

In the next result, we provide sufficient conditions for the weighted quasi-arithmetic mean to be (strongly) star quasiconvex.

\begin{theorem} {\bf (Star Quasiconvexity of \eqref{wqam}} \label{th:wqam}
 Let $X = \prod_{i=1}^n X_i$, where every $X_i \subseteq \mathbb{R}$ is convex, and $w_i > 0$ with $\sum_{i=1}^n w_i = 1$. Consider the weighted quasi-arithmetic mean $M_f: X \to \mathbb{R}$ defined as in \eqref{wqam}, where $f: I \to \mathbb{R}$ is a continuous strictly monotonic function on an interval $I \subseteq \mathbb{R}$ containing all values $x_i \in X_i$. Furthermore, let $\overline{x} = (\overline{x}_1,\dots,\overline{x}_n) \in {\rm argmin}_X M_f$ and $S(x) := \sum_{i=1}^n w_i f(x_i)$. Then the following assertions hold:
\begin{itemize}
 \item[$(a)$] Suppose that $f$ is increasing. If every function $h_i := f|_{X_i}$ is (strongly) star quasiconvex with respect to $\overline{x}_i$ with modulus $\gamma_i \geq 0$, then $M_f$ is (strongly) star quasiconvex with respect to $\overline{x}$ with modulus
 $$ \gamma_{M_f} = \mu \left( \min_{1 \leq i \leq n} w_i \gamma_i \right) \geq 0,$$
 where $\mu := \inf_{t \in S(X)} (f^{-1})^{D-}(t) \geq 0$.

 \item[$(b)$] Suppose that $f$ is decreasing. If every function $h_i := f|_{X_i}$ is (strongly) star quasiconcave with respect to $\overline{x}_i$ with modulus $\gamma_i \geq 0$, then $M_f$ is (strongly) star quasiconvex with respect to $\overline{x}$ with modulus
 $$\gamma^{\prime}_{M_f} = \mu' \left( \min_{1 \leq i \leq n} w_i \gamma_i \right) \geq 0,$$
 where $\mu' := \inf_{t \in (-S)(X)} \tilde{g}^{D-}(t) \geq 0$ and $\tilde{g}(t) = f^{-1}(-t)$.
\end{itemize}
\end{theorem}

\begin{proof}
 Proofs are similar, we only prove $(a)$. Since $f$ is increasing, $g := f^{-1}$ is increasing. Thus, $\overline{x} \in {\rm argmin}_X M_f$ if and only if $\overline{x} \in {\rm argmin}_X S$ by {\bf Fact 1}. 

 For every $i$, since $h_i := f|_{X_i}$ is (strongly) star quasiconvex with respect to $\overline{x}_i$ with modulus $\gamma_i \geq 0$, we have for every $y_i \in X_i$ and every $\lambda \in [0,1]$ that
 \begin{align*}
  & \hspace{1.35cm} h_i(\lambda \overline{x}_i + (1-\lambda)y_i) \leq h_i(y_i) - \lambda(1-\lambda)\frac{\gamma_i}{2}|y_i - \overline{x}_i|^2 \\
  & \overset{w_{i} > 0}{\Longrightarrow} ~ w_i h_i(\lambda \overline{x}_i + (1-\lambda)y_i) \leq w_i h_i(y_i) - \lambda(1-\lambda)\frac{w_i \gamma_i}{2}|y_i - \overline{x}_i|^2.
 \end{align*} 

 Let $\gamma_S := \min_{1 \leq i \leq n} w_i \gamma_i$. Since $w_i \gamma_i \geq \gamma_S$ for all $i$, it follows that
 $$w_i h_i(\lambda \overline{x}_i + (1-\lambda)y_i) \leq w_i h_i(y_i) - \lambda(1-\lambda)\frac{\gamma_S}{2}|y_i - \overline{x}_i|^2.$$

 Summing over $i = 1,\dots,n$, we get
 \begin{align*}
  S(\lambda \overline{x} + (1-\lambda)y) & \leq S(y) - \lambda(1-\lambda)\frac{\gamma_S}{2}\sum_{i=1}^n |y_i - \overline{x}_i|^2 \\
  & = S(y) - \lambda(1-\lambda)\frac{\gamma_S}{2} \|y - \overline{x}\|^2,
 \end{align*} 
 i.e., $S$ is (strongly) star quasiconvex with respect to $\overline{x}$ with modulus $\gamma_S = \min_{1 \leq i \leq n} w_i \gamma_i$.

 Since $g = f^{-1}$ is increasing, $\mu = \inf_{t \in S(X)} g^{D-}(t) \geq 0$. By Theorem \ref{prop:comp} applied to $g$ and $S$, $M_f = g \circ S$ is (strongly) star quasiconvex with respect to $\overline{x} \in {\rm argmin}_X M_f$ with modulus
 $$\gamma_{M_f} = \mu \gamma_S = \mu \left( \min_{1 \leq i \leq n} w_i \gamma_i  \right) \geq 0,$$
%
%
%
which completes the proof.
\end{proof}

In the star quasiconvex case (not strongly), we have the following characterization for star quasiconvexity of \eqref{wqam}.

\begin{theorem}{\bf (Characterization of Star Quasiconvexity of \eqref{wqam})} \label{th:wqam-converse}
 Let $X = \prod_{i=1}^n X_i$, where every $X_i \subseteq \mathbb{R}$ is convex, and $w_i > 0$ with $\sum_{i=1}^n w_i = 1$. Consider the weighted quasi-arithmetic mean $M_f: X \to \mathbb{R}$ defined as in \eqref{wqam}, where $f: I \to \mathbb{R}$ is a continuous strictly monotonic function on an interval $I \subseteq \mathbb{R}$ containing all values $x_i \in X_i$. Let $\overline{x} = (\overline{x}_1, \ldots, \overline{x}_n) \in {\rm argmin}_X M_f$. Then the following assertions hold:
 \begin{itemize}
  \item[$(a)$] Suppose that $f$ is increasing. Then $M_f$ is star quasiconvex with respect to $\overline{x}$ if and only if every function $f|_{X_i}$ is star quasiconvex with respect to $\overline{x}_i$.
     
  \item[$(b)$] Suppose that $f$ is decreasing. Then $M_f$ is star quasiconvex with respect to $\overline{x}$ if and only if every function $f|_{X_i}$ is star quasiconcave with respect to $\overline{x}_i$.
 \end{itemize}
\end{theorem}

\begin{proof}
 ($\Rightarrow$) Assume that $M_f$ is star quasiconvex with respect to $\overline{x}$. Fix an index $i_0 \in \{1,\ldots,n\}$ and take $y \in X_{i_0}$. Define $z = (z_1,\ldots,z_n) \in X$ by $z_j = \overline{x}_j$ for $j \neq i_0$ and $z_{i_{0}} = y$. Then, for every $\lambda \in [0,1]$, consider $\lambda \overline{x} + (1-\lambda)z$. Thus
 $(\lambda \overline{x} + (1-\lambda)z)_j = \overline{x}_j$ when $j \neq i_0$, and $(\lambda \overline{x} + (1-\lambda)z)_{i_{0}} = \lambda \overline{x}_{i_0} + (1-\lambda)y$. By star quasiconvexity of $M_f$ with respect to $\overline{x}$, we have
$$M_f(\lambda \overline{x} + (1-\lambda)z) \leq M_f(z).$$
Using the definition of $M_f$, this inequality is equivalent to
$$f^{-1} \left( w_{i_0} f(\lambda \overline{x}_{i_0} + (1-\lambda) y) + \sum_{j \neq i_0} w_j f(\overline{x}_j) \right) \leq f^{-1} \left( w_{i_0} f(y) + \sum_{j \neq i_0} w_j f(\overline{x}_j) \right).$$
Denote $C := \sum_{j \neq i_0} w_j f(\overline{x}_j)$. Then we have two cases:

\begin{itemize}
 \item[$(a)$] Suppose that $f$ is increasing. Then $f^{-1}$ is increasing, thus
\begin{align*}
 & f^{-1}\left( w_{i_0} f(\lambda \overline{x}_{i_0} + (1-\lambda)y) + C \right) \leq f^{-1}\left( w_{i_0} f(y) + C \right) \\
 & \Longrightarrow ~ w_{i_0} f(\lambda \overline{x}_{i_0} + (1-\lambda)y) + C \leq w_{i_0} f(y) + C \\
 & \Longrightarrow \hspace{1.35cm} f(\lambda \overline{x}_{i_0} + (1-\lambda)y) \leq f(y). 
\end{align*} 
Since $y \in X_{i_0}$ and $\lambda \in [0,1]$ were arbitrary, $f|_{X_{i_0}}$ is star quasiconvex with respect to $\overline{x}_{i_0}$. As $i_0$ was arbitrary, this holds for every $i \in \{1, \ldots, n\}$.

\item[$(b)$] Suppose that $f$ is decreasing. Then $f^{-1}$ is decreasing, so the inequality $M_f(\lambda \overline{x} + (1-\lambda)z) \leq M_f(z)$ implies
\[
w_{i_0} f(\lambda \overline{x}_{i_0} + (1-\lambda)y) + C \geq w_{i_0} f(y) + C ~ \Longrightarrow ~ f(\lambda \overline{x}_{i_0} + (1-\lambda)y) \geq f(y).
\]
This means $f|_{X_{i_0}}$ is star quasiconcave with respect to $\overline{x}_{i_0}$. Again, $i_0$ was arbitrary, so it holds for every $i \in \{1, \ldots, n\}$.
\end{itemize}

($\Leftarrow$) It is Theorem~\ref{th:wqam} for the 
case $\gamma = 0$.
\end{proof}

Now, we present a simple example of a weighted quasi-arithmetic mean that is star quasiconvex without being either quasiconvex or convex.

\begin{remark}
 Consider the weighted geometric mean on $[T_{1}, T_{2}] \times [T_{1}, T_{2}]$, $0 \leq T_{1} < T_{2} < + \infty$, given by
 $$ M(x_1, x_2) = \sqrt{x_1 x_2}.$$
 This is a quasi-arithmetic mean with $f(t) = \ln t$ and equal weights $w_1 = w_2 = \frac{1}{2}$. Observe that $M$ is star quasiconvex with respect to $(T_{1}, T_{1})$ by Theorem \ref{th:wqam}$(a)$, and it is not quasiconvex (thus nonconvex) since not all sublevel sets are convex.
\end{remark}

%

Means appears in modern financial applications, such as crypto-financial markets, as we illustrate below.

\begin{example}{\bf (Constant Functions Market Makers in Crypto)} \label{ex:crypto-markets}
 Automated Market Makers (AMMs) are algorithmic mechanisms that facilitate trading without traditional order books. In crypto-financial markets, AMMs are implemented as Decentralized Exchanges (DEXs) \cite{SN}.

 Let $n \geq 2$ be the number of cryptoassets in the DEX. A trade is represented by vectors $x, y \in \mathbb{R}^n_{+}$, where $x_i$ denotes the amount of asset $i$ the trader wishes to receive, and $y_i$ the amount offered in payment. 
 Thus, $(x,y) \in \mathbb{R}^n_{+} \times \mathbb{R}^n_{+}$ is a trade proposal.

The DEX holds reserves $R = (R_1,\ldots,R_n) \in \mathbb{R}^n_{+}$. A trading function $\varphi \colon \mathbb{R}^n_{+} \to \mathbb{R}$ determines whether a trade is admissible. Given a fee parameter $\gamma \in (0,1)$, a trade $(x,y)$ is accepted if and only if
\begin{equation}\label{valid:trades}
\varphi(R + \gamma y - x) = \varphi(R).
\end{equation}
In virtue of \eqref{valid:trades}, the literature also knows these models as {\rm Constant Function Market Makers} (see \cite{AB-2022,AC}).

Prominent DEX designs employ trading functions that are weighted quasi-arithmetic means. For example:
\begin{itemize}
    \item[$(i)$] Balancer \cite{MM}, Uniswap \cite{ZCP}, and SushiSwap \cite{AB-2022,Sushi} use the weighted geometric mean.
    \item[$(ii)$] StableSwap and mStable \cite{AB-2022,Egorov} use the weighted arithmetic mean.
\end{itemize}
Since these trading functions are instances of weighted quasi-arithmetic means (see \cite{ELS-1}), Theorems \ref{th:wqam} and \ref{th:wqam-converse} imply they are star quasiconvex under mild assumptions on the mean generator.
\end{example}

\subsection{Ratio Models with Multiple Risk Factors}
\label{subsec:4-2}

Let $X_i \subseteq \mathbb{R}^{n_i}$ and $Y_j \subseteq \mathbb{R}^{k_j}$ be convex sets for every $i \in \{1, \ldots, m\}$ and every $j \in \{1, \ldots, \ell\}$, respectively. Define $X := \prod_{i=1}^{m} X_i$ and $Y := \prod_{j=1}^{\ell} Y_j$, and consider two functions $f: X \rightarrow \mathbb{R}$ and $g: Y \rightarrow \, ]0, +\infty[$. We examine the ratio function $P: X \times Y \rightarrow \mathbb{R}$ given by
\begin{equation}\label{pro:function}
 P(x, y) = \frac{f(x)}{g(y)}.
\end{equation}
Such ratio structures arise naturally in models that trade off benefits against risks or costs, for instance, in efficiency analysis, risk-adjusted performance measurement, and preference models with background risk.  

A canonical and analytically convenient specification is obtained when \(f\) and \(g\) are multiplicatively separable. The Cobb-Douglas \cite{CD} form serves as an important illustration:
$$f(x) := c \prod_{i=1}^{m} x_i^{\alpha_i}, \qquad 
g(y) := d \prod_{j=1}^{\ell} y_j^{\beta_j},$$
with $c, d > 0$ and $\alpha_i, \beta_j > 0$. This specification admits two standard microeconomic interpretations:

\begin{itemize}
    \item[$(i)$] \textit{Production Theory.}  
    Here \(P(x,y)\) acts as a Shephard distance function \cite{McF}, measuring technical efficiency relative to a benchmark technology.  
    The numerator \(f(x)\) models output from inputs \(x\) (e.g., capital, labor), while the denominator \(g(y)\) captures production risk or opportunity costs influenced by factors \(y\) (e.g., input price volatility, policy uncertainty).

    \item[$(ii)$] \textit{Consumer Theory.}  
    In this setting \(P(x,y)\) models a risk‑adjusted uti\-li\-ty. The Cobb-Douglas term \(f(x)\) represents direct utility from goods \(x\), whereas \(g(y)\) adjusts for risk or disutility associated with factors \(y\) (e.g., investment risk, price fluctuations).  
    Such a form neatly captures benefit–cost (or benefit-risk) trade‑offs in a normalized way \cite{ADSZ}.
\end{itemize}

To analyze optimization problems involving \(P\), it is standard to apply a monotonic transformation that preserves preference or technology rankings.  Choosing the logarithm  yields the additively separable form (for simplicity, we set $n_i = k_j = 1$; the multidimensional case follows by the same argument)
$$(\ln \circ P)(x, y) = \ln(c) - \ln(d) + \sum_{i=1}^{m} \alpha_i \ln(x_i) - \sum_{j=1}^{\ell} \beta_j \ln(y_j).$$
In a minimization context, one considers
\begin{equation}\label{eq:motiv}
 \min_{x, y} \Bigl( \ln(d) - \ln(c) + \sum_{j=1}^{ \ell} \beta_j \ln(y_j) - \sum_{i=1}^{m} \alpha_i \ln(x_i) \Bigr).
\end{equation}

The function in \eqref{eq:motiv} is convex (hence quasiconvex) in $x$ and concave (but also quasiconvex) in $y$. To maintain classical quasiconvexity of the separable sum, the Debreu-Koopmans condition would force $\ell = 1$ (see also \cite[Lemma 5.11]{ADSZ}).  That restriction would allow only a single risk variable $y$, an unrealistic limitation in virtually any applied economic setting.

On the other hand, each component of \eqref{eq:motiv} is \textit{star quasiconvex} since the $\ln(x_i)$ terms are convex (hence star quasiconvex), and the $\ln(y_j)$ terms are concave and quasiconvex (thus also star quasiconvex).  Consequently, by Theorem \ref{th:sepa}, the entire separable function $(\ln \circ P) (x, y)$ is star quasiconvex.  

Since the logarithmic transformation is increasing and continuous, and function $\ln P(x,y)$ is star quasiconvex by Theorem \ref{th:sepa}, it follows by Theorem \ref{prop:comp} that $P(x,y) = \exp(\ln P(x,y))$ is also star quasiconvex. Thus, the original economic model $P(x,y)$ represents preferences (or technology) consistent with the  economically meaningful structure of star‑shaped preferences.  

Consequently, star quasiconvexity permits the model to incorporate \textit{multiple risk factors} ($\ell \geq 1$) without violating the diversification axiom, thereby solving the limitation imposed by the classical Debreu-Koopmans theorem.  The same reasoning applies to more flexible separable specifications (e.g., CES or translog forms), underscoring the broad applicability of the star quasiconvex framework in modern economic modeling.


\section{Conclusions and Future Work}\label{sec:05}

This paper solved the Debreu-Koopmans problem by showing that the separable aggregation of quasiconvex components yields a star quasiconvex function, whose sublevel sets are star-shaped with respect to minimizers. This result bridges classical utility theory and behavioral economics: it permits separable functional forms to incorporate multiple nonconvex components while preserving diversification, thereby overcoming the restrictive {\it single nonconvex component} limitation of classical quasiconvexity.


Furthermore, we have developed a comprehensive calculus for star quasiconvexity. This calculus enables the construction of economically relevant models that were previously excluded by the Debreu-Koopmans constraint, demonstrating that star quasiconvexity is not merely a weaker property but a mathematically rich structure with its own calculus.

Looking forward, our findings open several promising directions:
\begin{itemize}
    
 \item[$(i)$] {\it Applied Economic Modeling}: Employing star quasiconvex separable forms in integrated assessment models, financial risk frameworks, and production models where classical quasiconvexity restrictions are prohibitive.
    
 \item[$(ii)$] {\it Computational Methods}: Exploiting the separable structure of star quasiconvex objectives to develop efficient optimization algorithms for large-scale economic and machine learning problems, potentially leveraging parallelization and decomposition techniques.
\end{itemize}

From a mathematical point of view, the Debreu-Koopmans problem remains open when the aggregate function $h$ has no minimizers, for instance, when dealing with unbounded domains or preferences that lack a satiation point. This issue is both significant and technically challenging, requiring a deeper understanding of star quasiconvexity on noncompact sets. For instance, observe that the nonconvex function 
$$h(x_{1}, x_{2}) = \ln (x_{1}) + \ln(x_{2}),$$ 
is a kind of ``star quasiconvex function" at $(0, 0) \in \mathbb{R}^{2}$, even when it is not defined at that point.


\section{Declarations}\label{sec:06}







\subsection{Availability of Supporting Data}

No data sets were generated during the current study. 

%

\subsection{Competing Interests}

There are no conflicts of interest or competing interests related to this manuscript.

\subsection{Funding}

This research was partially supported by ANID--Chile under project Fondecyt Regular 1241040. 


\end{document}